\documentclass[a4paper,12pt,reqno,oneside]{amsart}
\usepackage{setspace} 
\usepackage{typearea} 

\usepackage{amscd,amsmath,amssymb,verbatim}
\usepackage{xr-hyper}
\usepackage{graphicx}
\usepackage{ifthen}
\usepackage{paralist,cite}
\usepackage[colorlinks,backref,final,bookmarksnumbered,bookmarks]{hyperref}
\usepackage[backrefs]{amsrefs}
\usepackage{enumitem}

\usepackage[utf8]{inputenc}
\usepackage[T1,T2A]{fontenc}

\newcommand{\arxiv}[2][]{\ifthenelse{\equal{#1}{}}
{\href{http://arxiv.org/abs/#2}{\tt arXiv:#2}}
{\href{http://arxiv.org/abs/math/#2}{\tt arXiv:math.#1/#2}}}

\theoremstyle{plain}
\newtheorem{theorem}{Theorem}[section]
\newtheorem{lemma}[theorem]{Lemma}
\newtheorem{corollary}[theorem]{Corollary}
\newtheorem{proposition}[theorem]{Proposition}

\theoremstyle{definition}
\newtheorem{example}[theorem]{Example}

\newtheoremstyle{remark}
{}{}{}{}{\itshape}{}{ }{\thmname{#1}\thmnumber{ \itshape #2.}}
\theoremstyle{remark}
\newtheorem{remark}[theorem]{Remark}

\newtheoremstyle{concise}
{}{}{}{}{\bfseries}{}{ }{\thmnumber{#2.}\thmnote{ #3.}}
\theoremstyle{concise}

\def\N{\mathbb{N}} 
\def\R{\mathbb{R}} 
\def\Z{\mathbb{Z}}

\def\x{\times}
\def\but{\setminus} 
\def\emb{\hookrightarrow} 
\def\i{\subset}
\def\C{\mathcal{C}}
 
\def\phi{\varphi}

\def\xr#1{\xrightarrow{#1}} 
 \renewcommand{\:}{\colon}

\DeclareMathOperator{\cyl}{cyl}
\DeclareMathOperator*{\colim}{colim} 
\DeclareMathOperator*{\derlim}{lim^1}

\DeclareMathOperator{\Int}{Int} 
 \DeclareMathOperator{\Cl}{Cl} 
\DeclareMathOperator{\Fr}{Fr} \DeclareMathOperator{\Hom}{Hom} 
 
\DeclareMathOperator{\supp}{supp} \DeclareMathOperator{\Ext}{Ext}

\def\tph#1{\raise2.5pt\hbox{\the\textfont1\char"7F}\!\!#1}
\def\tpm#1{\raise0pt\hbox{\the\textfont1\char"7F}\!#1}
\def\tpl#1{\lower1.5pt\hbox{\the\textfont1\char"7F}\!#1}

\begin{document}

\title[Algebraic topology of Polish spaces. II: Axiomatic homology]
{Algebraic topology of Polish spaces. \\ II: Axiomatic homology}
\author{Sergey A. Melikhov}
\date{\today}
\address{Steklov Mathematical Institute of Russian Academy of Sciences,
ul.\ Gubkina 8, Moscow, 119991 Russia}
\email{melikhov@mi.ras.ru}

\begin{abstract}
Milnor proved two uniqueness theorems for axiomatic (co)homology: one for pairs of compacta (1960) 
and another, in particular, for pairs of countable simplicial complexes (1961). 
We obtain their common generalization: the Eilenberg--Steenrod axioms along with Milnor's map excision axiom and 
a (non-obvious) common generalization of Milnor's two additivity axioms suffice to uniquely characterize 
(co)homology of closed pairs of Polish spaces (=separable complete metrizable spaces).
The proof provides a combinatorial description of the (co)homology of a Polish space in terms of a cellular
(co)chain complex satisfying a symmetry of the form $\lim\,\colim = \colim\,\lim$. 
\end{abstract}

\maketitle
\section{Introduction}

In this paper we obtain an axiomatic characterization of Steenrod--Sitnikov homology and \v Cech cohomology
on closed pairs of Polish spaces.
The only new axiom that we need is the following 

\theoremstyle{plain}
\newtheorem*{CAA}{Controlled Additivity Axiom}
\newtheorem*{UT}{Uniqueness Theorem}

\begin{CAA} If $X$ is a closed subset of a Polish space $M$ such that $M\but X$ is represented as 
a disjoint union $\bigsqcup_{i\in\N}K_i$, where the diameters of the $K_i$ tend to zero whenever they approach $X$, 
then $H_n(M,X)$ is isomorphic in a natural way to a certain explicit subgroup of $\prod_{i\in\N} H_n(K_i)$, 
which contains $\bigoplus_{i\in\N} H_n(K_i)$; and similarly for $H^n(M,X)$.
\end{CAA}

The subgroup in question is the so-called $\bold K$-direct sum, where $\bold K$ is a certain explicit ideal in 
the boolean algebra of all subsets of the indexing set $\N$.
This ideal is nothing but an obvious incarnation of the dual pair of filtrations introduced in Part I of
the present series \cite{M-I}*{\S\ref{fish:filtrations}}.
The full details of the statement of the axiom are in \S\ref{controlled additivity}.

The Controlled Additivity Axiom is a common generalization of Milnor's two additivity axioms.
The case $X=\emptyset$ is nothing but additivity with respect to countable disjoint union.
The case where $X$ is a single point and $M$ is compact is precisely Milnor's Cluster Axiom, that is,
additivity with respect to countable metric wedge.

\begin{UT} 
If $H$ and $H'$ are ordinary (co)homology theories (i.e., functors satisfying the Eilenberg--Steenrod axioms) 
on the category of closed pairs of Polish spaces, which additionally satisfy the Wallace--Milnor Map Excision 
Axiom and the Controlled Additivity Axiom, then any graded isomorphism $H(pt)\to H'(pt)$ extends to 
a natural equivalence from $H$ to $H'$.
\end{UT}

The proof is in \S\ref{uniqueness-chapter}.
The necessary background, along with a brief overview of previously known axiomatic characterizations of 
homology and cohomology, is provided in \S\ref{background}.

The main point of the Uniqueness Theorem seems to be not philosophical but computational.
That Steenrod--Sitnikov homology and \v Cech cohomology are the ``right'' theories (at least for the purposes
of geometric topology, and at least for metrizable spaces) was more or less clear anyway, from previously known 
axiomatic characterizations (see \S\ref{general-uniqueness}) and from other results.
On the other hand, an axiomatic characterization can also be understood, at least in theory, as a computational 
tool, enabling one to compute homology and cohomology of an arbitrary space or pair (in a certain generality) 
directly from the stated axioms. 
In this respect, usefulness of previously known axiomatic characterizations on categories including Polish
spaces is somewhat questionable (see \S\ref{general-uniqueness}), and our Uniqueness Theorem does seem to be 
a major advance.

Namely, in the course of proving the Uniqueness theorem we express the Steenrod--Sitnikov homology and 
the \v Cech cohomology of a Polish space as (co)homology of a certain cellular (co)chain complex, based
on (co)chains ``that are finite in one direction but possibly infinite in the other direction''
(\S\ref{chain complexes}).
This (co)chain complex admits two different descriptions: as a direct limit of inverse limits and
as an inverse limit of direct limits.
The fact that they lead to the same (co)chain complex suggests that it may be a noteworthy object in 
its own right, and we devote the final section (\S\ref{chain complexes2}) to a clarification of its 
geometric structure.
In fact, this chain complex has already been used to establish one result in a subsequent paper 
\cite{M-III}*{Theorem \ref{lim:lim2theorem}}.

This paper can be read independently of Part I \cite{M-I}; the only result of Part I that will be essentially
used here is the duality lemma \cite{M-I}*{Lemma \ref{fish:duality}}.
This lemma does, however, provide a strong conceptual connection between the central ideas of Part I and
of the present Part II: it is the ultimate reason why ``inverse and direct limits commute'' in both papers
(in contrast to the subsequent paper \cite{M-III}, in which they ``do not commute'').

\section{Background} \label{background}

\subsection{Eilenberg--Steenrod axioms}
Let us recall the Eilenberg--Steenrod definition of homology and cohomology theories \cite{ES}.
Most algebraic topology textbooks give it only for one specific category (the category of pairs of 
topological spaces).
Eilenberg and Steenrod allow any ``admissible'' category, whose definition is a bit long.
But we will need only a smaller class of categories, which is easy to describe.

Let $\C$ be a full subcategory of the category of pairs of topological spaces which contains $\emptyset$, all
singleton spaces and $I:=[0,1]$ (we identify every space $X$ with the pair $(X,\emptyset)$) and is closed 
under finite products and under the functors $(X,A)\mapsto X$, $(X,A)\mapsto A$ and $X\mapsto (X,X)$.

A {\it (co)homology theory} on $\C$ consists of
\begin{itemize}
\item a covariant (resp.\ contravariant) functor $H=(H_n)$ (resp.\ $H=(H^n)$) from $\C$ to the category of 
$\Z$-graded abelian groups (where $H_n(f)$, resp.\ $H^n(f)$ is usually abbreviated as $f_*$, resp.\ $f^*$);
\item a natural transformation $\partial=(\partial_n)$ of degree $-1$ from $H$ to the composition of 
$(X,A)\mapsto A$ with $H$ (resp.\ a natural transformation $d=(d_n)$ of degree $1$ from the composition 
of $(X,A)\mapsto A$ with $H$ to $H$)
\end{itemize}
provided that the following axioms are satisfied:
\begin{itemize}
\item (exactness) if $i\:A\emb X$ and $j\:X\emb (X,A)$ are in $\C$, then
\[\dots\to H_n(A)\xr{i_*} H_n(X)\xr{j_*} H_n(X,A)\xr{\partial_n} H_{n-1}(A)\to\dots,\]
respectively, \[\dots\to H^n(X,A)\xr{j^*} H^n(X)\xr{i^*} H^n(A)\xr{d_n} H^{n+1}(X,A)\to\dots\] 
is exact;
\item (homotopy) if $f,g\:(X,A)\to (Y,B)$ in $\C$ are homotopic, then $H(f)=H(g)$;
\item (excision) if $u\:(X\but U,\,A\but U)\emb(X,A)$ is in $\C$, where $U$ is open in $X$ and $\Cl U\subset\Int A$,
then $H(u)$ is an isomorphism.
\end{itemize}
The (co)homology theory is called {\it ordinary} if it additionally satisfies 
\begin{itemize}
\item (dimension) if $X$ is a singleton, then $H_i(X)=0$ (resp.\ $H^i(X)=0$) for all $i\ne 0$.
\end{itemize}
In this case $H_i(pt)$ (resp.\ $H^i(pt)$), where $pt=\{\emptyset\}$, is called the {\it coefficient} group
of the ordinary theory.

Eilenberg and Steenrod proved that an ordinary (co)homology theory on the category $\C_0$ of pairs of 
finite simplicial complexes and their subcomplexes is {\it unique} (with given coefficients), in the sense 
that for any two ordinary (co)homology theories $H$, $H'$ on $\C_0$, any graded isomorphism $H(pt)\to H'(pt)$ 
extends to a natural equivalence from $H$ to $H'$ \cite{ES}*{Theorem III.10.1}.

(Co)homology theories on the category $\C_\infty$ of pairs of metrizable spaces are not unique.
In particular, they include both Steenrod--Sitnikov homology (see \cite{Mas}, \cite{Sk5}) and singular homology;
and both \v Cech cohomology (see \cite{Sp}, \cite{Sk5}) and singular cohomology.

\subsection{Milnor's axioms}\label{Milnor axioms}

The {\it cluster} $\bigvee_{i\in\N}(X_i,x_i)$ of pointed metrizable spaces $(X_i,x_i)$ (not to be confused 
with their non-metrizable infinite wedge nor with the metric wedge of \cite{M-III}*{\S\ref{lim:metric quotient}}) 
is the limit of the inverse sequence of finite wedges, $\dots\xr{p_2} (X_1,x_1)\vee (X_2,x_2)\xr{p_1} (X_1,x_1)$, 
where each $p_n$ shrinks the $(n+1)$st factor of the wedge $\bigvee_{i\in\N}(X_i,x_i)$ onto the wedge point.
Here the finite wedges are regarded as pointed spaces, and so is their inverse limit.
We sometimes abbreviate $\bigvee_i(X_i,x_i)$ by $\Big(\bigvee_i X_i,\,\infty\Big)$.

Alternatively, $\bigvee_i X_i$ may be defined as the subset of $\prod_i X_i$ 
consisting of all points with all coordinates except possibly one equal to the basepoint, with basepoint 
$\infty:=(x_1,x_2,\dots)$.

The following axioms for a (co)homology theory on $\C$ are due to Milnor \cite{Mi1}, \cite{Mi2}:
\begin{itemize}
\item ($\bigsqcup$-additivity) if $\bigsqcup_{i\in\N} X_i$ and each $X_i$ are in $\C$, then
for each $n$ the inclusions $X_k\emb\bigsqcup_i X_i$ induce an isomorphism 
\[H_n\Big(\bigsqcup_{i\in\N} X_i\Big)\simeq\bigoplus_{i\in\N} H_n(X_i),\]
respectively, 
\[H^n\Big(\bigsqcup_{i\in\N} X_i\Big)\simeq\prod_{i\in\N} H^n(X_i);\] 
\item ($\bigvee$-additivity) if $\Big(\bigvee_{i\in\N} X_i,\,\infty\Big)$ and each $(X_i,x_i)$ are 
in $\C$, then for each $n$ the retraction maps $\Big(\bigvee_i X_i,\,\infty\Big)\to (X_k,x_k)$ induce 
an isomorphism 
\[H_n\Big(\bigvee_{i\in\N} X_i,\,\infty\Big)\simeq\prod_{i\in\N} H_n(X_i,x_i),\]
respectively, 
\[H^n\Big(\bigvee_{i\in\N} X_i,\,\infty\Big)\simeq\bigoplus_{i\in\N} H^n(X_i,x_i);\]
\item (map excision) if $f\:(X,A)\to(Y,B)$ is in $\C$, where $A$, $B$ are closed in $X$, $Y$ and $f$ is 
a closed map that restricts to a homeomorphism between $X\but A$ and $Y\but B$, 
then for each $n$, $f$ induces an isomorphism
\[H_n(X,A)\simeq H_n(Y,B),\]
respectively,
\[H^n(X,A)\simeq H^n(Y,B).\]
\end{itemize}

Milnor proved the uniqueness of $\bigvee$-additive (co)homology satisfying map excision on the category 
$C_\kappa$ of pairs of compact metrizable spaces \cite{Mi1}.
He also proved the uniqueness of $\bigsqcup$-additive (co)homology on the category of pairs of countable 
CW-complexes and their subcomplexes \cite{Mi2}. 
Similar arguments establish the uniqueness of $\bigsqcup$-additive (co)homology on the category $\C_\nu$ 
of pairs of countable simplicial complexes and their subcomplexes with metric topology; for our purposes
it is convenient to use the metric defined in \cite{M3}.

Petkova (refining previous work of Sklyarenko \cite{Sk71}, who used a different set of axioms) proved that 
Steenrod--Sitnikov homology and \v Cech cohomology are unique ordinary homology theories on the category 
of closed pairs of {\it locally compact} separable metrizable spaces that satisfy $\bigvee$-additivity 
and map excision on pairs of {\it compacta}, and $\bigsqcup$-additivity for disjoint unions of 
compacta \cite{Pe2} (homology), \cite[Theorem 9]{Pe1} (cohomology).

\subsection{Remarks on Milnor's axioms}

To be precise, Milnor formulated the $\bigvee$-additivity and the map excision axioms only for one specific 
category (namely, $\C_\kappa$).
However, the validity of the map excision axiom for the Alexander--Spanier cohomology (for metrizable spaces),
had been previously proved by Wallace \cite{Wal} (see also \cite{Sp}*{Theorem 6.6.5}).
The map excision axiom is also valid for Steenrod--Sitnikov homology (for metrizable spaces) 
\cite{Mas}*{Theorem 9.7}.
The restriction to closed maps is essential in the map excision axiom, by considering the inclusion 
$([0,1),\{0\})\emb ([0,1],\{0,1\})$ (see details in \cite{Mas}*{remark after Theorem 8.7}; see also
Example \ref{non-closed} below.

It was shown in Part I of the present series \cite{M-I} that if a (co)homology theory on closed pairs 
of metrizable spaces is fine shape invariant, then it satisfies map excision.
Since \v Cech cohomology and Steenrod--Sitnikov homology are fine shape invariant \cite{M-I}, this yields 
an alternative, short proof that they satisfy map excision.   
By a result of Mrozik, for a (co)homology theory on pairs of compacta, map excision is
equivalent to invariance under fine shape (see \cite{M-I}).

The $\bigoplus$-additivity of Steenrod--Sitnikov homology and \v Cech cohomology 
(i.e.\ the $\bigsqcup$-additivity for homology and the $\bigvee$-additivity for cohomology,
for metrizable spaces) are obvious from the definitions.

\begin{theorem}\label{10.0} Steenrod--Sitnikov homology and \v Cech cohomology are 
$\prod$-additive (on metrizable spaces).
\end{theorem}

The $\bigvee$-additivity of Steenrod--Sitnikov homology is well-known for compacta (cf.\ \cite{Mi2}), 
and the general case is addressed below.
The $\bigsqcup$-additivity of \v Cech cohomology is well-known \cite{Sp} and can be proved 
similarly.

\begin{proof} We need to show that $H_q(\bigvee_{m\in\N} X_m,\infty)\simeq\prod_{m\in\N} H_q(X_m,x_m)$.
Let us assume that the $X_m$ are locally compact; this simplifies notation but does not affect 
the essence of the argument.
In this case, it suffices to verify that if for each $m$ we are given a direct sequence $G_{m1}\to G_{m2}\to\dots$ 
of abelian groups, then the natural map 
\[\phi\:\colim\limits_{f\in\N^\N}\prod_{m\in\N}G_{m,f(m)}\to\prod_{m\in\N}\colim\limits_{n\in\N} G_{mn}\]
(sending the class of a sequence $(x_1,x_2,\dots)\in\prod_m G_{m,f(m)}$ into the sequence of the classes of 
the $x_i$, which is obviously well-defined) is an isomorphism.
(Compare \cite{HoTT}*{2.5.17}.)

Let us construct its inverse.
Let $[x_{mn}]$ be the class in $\colim_n G_{mn}$ of an element $x_{mn}\in G_{mn}$, and let 
$([x_{1,f(1)}],[x_{2,f(2)}],\dots)$ be a sequence of such classes for some $f\in\N^\N$.
Suppose that each $[x_{m,f(m)}]=[y_{m,g(m)}]$ for some $g\in\N^\N$.
Then $x_{m,f(m)}$ and $y_{m,g(m)}$ have equal images in some $G_{m,h(m)}$.
Here $h(m)\ge\max\big(f(m),g(m)\big)$ for each $m$, so $h\ge f$ and $h\ge g$.
Hence $(x_{1,f(1)},x_{2,f(2)},\dots)$ and $(y_{1,g(1)},y_{2,g(2)},\dots)$ have equal images in 
$\prod G_{m,h(m)}$ and thus determine the same element of the left hand side. 
\end{proof}

A cluster of the form $\bigvee_{i\in\N}(X_i\sqcup\{y_i\},y_i)$ is called the {\it null-sequence}
of the $X_i$ and will be denoted $\Big(\bigvee^+_{i\in\N} X_i,\,\infty\Big)$.
In the case of null-sequences, $\bigvee$-additivity specializes to:
\begin{itemize}
\item ($\bigvee^+$-additivity) if $\Big(\bigvee^+_{i\in\N} X_i,\,\infty\Big)$ and each $X_i$ are in $\C$, 
then for each $n$ the retraction maps $\Big(\bigvee^+_i X_i,\,\infty\Big)\to (X_k\sqcup\{\infty\},\infty)$ 
induce an isomorphism 
\[H_n\Big(\bigvee^+_{i\in\N} X_i,\,\infty\Big)\simeq\prod_{i\in\N} H_n(X_i),\]
respectively, 
\[H^n\Big(\bigvee^+_{i\in\N} X_i,\,\infty\Big)\simeq\bigoplus_{i\in\N} H^n(X_i).\]
\end{itemize}
In the presence of map excision, $\bigvee^+$-additivity is in fact equivalent to $\bigvee$-additivity.
Indeed, $\bigvee^+$-additivity implies its relative version using the five-lemma, and it remains to
consider the quotient map $\bigvee^+_i X_i\to\bigvee^+_i X_i/\bigvee^+_i\{x_i\}\cong\bigvee_i(X_i,x_i)$.
For a (co)homology theory on pairs of compacta, map excision is obviously equivalent to strong excision: 
for every compact pair $(X,A)$ the quotient map $(X,A)\to (X/A,\,pt)$ induces an isomorphism on (co)homology.

Singular homology and cohomology do not satisfy $\bigvee$-additivity on compacta (see \cite{M1}*{Example 5.6}).
They also do not satisfy map excision on pairs of compacta, since they are not invariants of fine shape 
(see \cite{M1}).

\subsection{What is a uniqueness theorem?}\label{general-uniqueness}
By adding to the Eilenberg--Steenrod and Milnor axioms some rather natural assumptions, one can in fact obtain
uniqueness theorems (some of them rather trivial) for fairly general categories of pairs.

An ordinary cohomology theory on pairs of polyhedra admits a unique ``continuous'' (in the sense of 
Mare\v si\' c resolutions) extension to all topological spaces and their P-embedded%
\footnote{A subset $A\subset X$ is called {\it P-embedded} if every continuous pseudometric on $A$
can be extended to a continuous pseudometric on $X$.
All subsets of $X$ are P-embedded if and only if $X$ is collectionwise normal \cite{Sh}.
(Collectionwise normal spaces include all paracompact Hausdorff spaces, which in turn include all metrizable
spaces.)
A subset $A\subset X$ is P-embedded if and only if there exists a polyhedral resolution of $(X,A)$
that restricts to a polyhedral resolution of $A$ \cite{MS}*{pp.\ 89--90}).}
subsets \cite{Wa}*{Theorem 7}.
By adding this continuity axiom to Milnor's axioms for the category of polyhedral pairs, one obtains
an axiomatic characterization of \v Cech cohomology on all topological spaces and their P-embedded subsets
\cite{Wa}*{Corollary 8(ii)}.
Alternatively, \v Cech cohomology is the only cohomology theory on paracompact Hausdorff spaces and their 
closed subsets that satisfies $\bigsqcup$-additivity, the vanishing of homology in negative dimensions, and 
a local vanishing condition: $\colim H^n(U,pt)=0$ over all neighborhoods $U$ of any fixed $pt\in X$ \cite{Bac}.

An ordinary homology theory on pairs compacta obviously admits a unique extension to all metrizable spaces and 
their closed subsets that is ``compactly supported'' in the sense that $H_n(X)=\colim H_n(K)$ over all compact 
$K\subset X$.
By adding this axiom of compact supports to Milnor's axioms for the category of pairs of compacta, one obtains
an axiomatic characterization of Steenrod--Sitnikov homology on metrizable spaces and their closed subsets
\cite{Pe1}*{Theorem 7}.
One can drop the restriction of metrizability by starting from an axiomatic characterization of homology
on pairs of compact Hausdorff spaces.
In fact, there do exist some fairly reasonable uniqueness theorems for homology theories on compact 
Hausdorff pairs.
Uniqueness is provided
\begin{itemize}
\item by the ``dual'' universal coefficient formula with respect to \v Cech cohomology,
$0\to\Ext\big(H^{n+1}(X),\,G)\to H_n(X;\,G)\to\Hom\big(H^n(X),\,G)\to 0$ \cite{Ber}; 
\item by strong excision, $\bigvee$-additivity and the vanishing of homology on all pairs that are acyclic with
respect to \v Cech cohomology (Berikashvili; see \cite{Sk5}*{\S7.2.4}); 
\item by strong excision, the Vietoris--Begle theorem, a restricted continuity axiom (for inverse limits of
finite wedges of spheres), and the vanishing of homology in negative dimensions \cite{Ber}; 
\item by the Milnor-type short exact sequence, cf.\ \cite{M-III}*{Proposition \ref{lim:milnor-ses}, 
Remark \ref{lim:lim2seq}} (Inasaridze--Mdzinarishvili; see \cite{BM}, \cite{Sk5}*{\S7.2.4}).
\end{itemize}

A considerable number of ordinary homology and cohomology theories actually constructed in the literature 
over the years eventually turned out to satisfy the above axioms, each time in a generality that appears 
to be natural for the theory in question (see \cite{Sk5}, \cite{Ber}, \cite{Sk6}).

Apart from the case of compact Hausdorff pairs, the above uniqueness theorems involve direct limits of 
abelian groups over uncountable directed sets.
Being an exact functor, direct limit even over a large directed set may seem harmless from a philosophical 
viewpoint.
This may be enough for some to philosophically accept the homology theory and the cohomology theory 
characterized by the above axioms (i.e., \v Cech cohomology and the extension of Steenrod--Sitnikov 
homology beyond metrizable spaces) as the ``right'' ones.
But from a computational viewpoint (and also from the viewpoint of logical complexity), direct limit is 
certainly far more complicated than $\bigoplus$ or $\prod$, since we are given no information about 
the bonding maps.
So if we want a ``constructive'' uniqueness theorem, which would provide a practical tool for computing 
a given (co)homology theory, then we definitely need a new, far more explicit characterization of
(co)homology beyond polyhedra and locally compact spaces.

\section{Controlled additivity axiom} \label{controlled additivity}

\subsection{The axiom}

If $M$ is a metrizable space and $X\subset M$ is closed, let us call a family of subsets $Y_i$ of 
$M\but X$ {\it scattered towards} $X$ if for any sequence of points $y_i\in Y_{n_i}$ that converges to an $x\in X$,
any of the following equivalent conditions holds:
\begin{itemize}
\item the diameters of the $Y_{n_i}$ tend to zero as $i\to\infty$, for some metric on $M$;
\item the diameters of the $Y_{n_i}$ tend to zero as $i\to\infty$, for every metric on $M$;
\item any other sequence of points $y'_i\in Y_{n_i}$ also converges to $x$.
\end{itemize} 

For the remainder of this subsection, let us assume that
\begin{enumerate}
\item $M$ is a metrizable space;
\item $X$ is a closed subspace of $M$; 
\item $M\but X$ is a union of its pairwise disjoint subspaces $C_i$, $i\in\N$, and moreover has the topology of the
disjoint union $\bigsqcup_{i\in\N} C_i$;
\item the $C_i$ are scattered towards $X$.
\end{enumerate}

Let us note that each $C_n$ is clopen not only in $M\but X$, but also in $M$.
Indeed, $C_n$ is open in $M$ due to (2).
If some sequence of points $y_i\in C_n$ converges to a point $y\notin C_n$, then $y\in X$ by (3).
But then the diameter of $C_n$ must be zero by (4).
Hence each $y_i=y_1$, but then also $y=y_1$, contradicting our hypothesis $y\notin C_n$.

Let $M_\odot$ be the quotient space of $M$ obtained by shrinking each $C_i$ to a point.
Since the indexing of the $C_i$ is assumed fixed, we can identify $M_\odot\but X$ with $\N$.
Clearly, any section $M_\odot\to M$ of the quotient map $M\to M_\odot$ is an embedding.

Let 
\[\kappa=\{K\subset\N\mid K\text{ has compact closure in }M_\odot\},\] 
\[\nu=\{U\subset\N\mid U\cup X\text{ is open in }M_\odot\}.\]
Of course, these are precisely the compact filtration and the neighborhood cofiltration 
of \S\ref{filtrations}.
Like before, instead of $\nu$ it is often more convenient to use the co-neighborhood filtration
\[\bar\nu=\{F\subset\N\mid F\text{ is closed in }M_\odot\}.\]
Clearly, $\kappa$ and $\bar\nu$ are ideals in the boolean algebra of all subsets of $\N$.

The controlled additivity axiom asserts that certain natural maps induce for each $n$ an isomorphism
\[H_n(M,X)\simeq\colim_{K\in\kappa}\prod_{i\in K}H_n(C_i)=\bigcup_{K\in\kappa}\prod_{i\in K}H_n(C_i)
\subset\prod_{i\in\N}H_n(C_i),\]
respectively,
\[H^n(M,X)\simeq\colim_{F\in\bar\nu}\prod_{i\in F}H^n(C_i)=\bigcup_{F\in\bar\nu}\prod_{i\in F}H^n(C_i)
\subset\prod_{i\in\N}H^n(C_i).\]
The two unions of products are known as ``$\bold K$-direct sums'' (where $\bold K=\kappa$, $\bar\nu$)
in the theory of abelian groups \cite{Fu}*{\S8}.

In more detail, for each $S\subset\N$ let us write $C_S=\bigcup_{i\in S}C_i$ and $\bar C_S$ for the closure 
of $C_S$ in $M$.
The inclusions $(\bar C_K,\,\bar C_K\cap X)\emb\big(M,X)$, where $K\in\kappa$, induce a map
\[\colim_{K\in\kappa} H_n(\bar C_K,\,\bar C_K\cap X)\xr{\phi_*}H_n(M,X).\]
Also, the retractions $(\bar C_K,\,\bar C_K\cap X)\to(C_i\sqcup pt,pt)$, where $i\in K$, induce a map
\[H_n(\bar C_K,\,\bar C_K\cap X)\xr{\phi^K_*}\prod_{i\in K}H_n(C_i).\]
Dually, the retractions $(M,X)\to(C_F\sqcup pt,pt)$, where $F\in\bar\nu$, induce a map
\[\colim_{F\in\bar\nu} H^n(C_F)\xr{\psi^*} H^n(M,X).\]
Also, the inclusions $C_i\emb C_F$, where $i\in F$, induce a map
\[H^n(C_F)\xr{\psi^*_F}\prod_{i\in F}H^n(C_i).\]

The controlled additivity axiom for a (co)homology theory $H$ on $\C$ requires that each pair 
$(\bar C_S,\,\bar C_S\cap X)$ be in $\C$, and that the maps $\phi_*$ and $\phi^K_*$, $K\in\kappa$
(respectively, $\psi^*$ and $\psi^*_F$, $F\in\bar\nu$) be isomorphisms.

\begin{example} When $X=\emptyset$, the controlled additivity axiom turns into $\bigsqcup$-additivity.
When $X=pt$ and $M_\odot$ is compact, the controlled additivity axiom turns into $\bigvee^+$-additivity.
\end{example}

\begin{proposition} Steenrod--Sitnikov homology and \v Cech cohomology satisfy the controlled 
additivity axiom on closed pairs of metrizable spaces.
\end{proposition}

\begin{proof} That all $\phi_*^K$ and $\psi^*_F$ are isomorphisms is Milnor's $\prod$-additivity axioms
(along with strong excision, in the case of $\phi_*^K$).

Next, by Lemma \ref{cofinals}(b), every compact subset of $M$ lies in $C_K\cup L$ for some $K\in\kappa$ and 
some compact $L\subset X$.
The inclusion $(\bar C_K,\,\bar C_K\cap X)\emb\big(\bar C_K\cup L,\,(\bar C_K\cup L)\cap X\big)$ induces 
an isomorphism on homology by the map excision axiom.
Therefore the definition of Steenrod--Sitnikov homology implies that $\phi_*$ is an isomorphism.

Finally, by Lemma \ref{cofinals}(a), every neighborhood of $X$ contains $X\cup C_U$ for some $U\in\nu$.
Therefore by Spanier's tautness theorem (see \cite{Sp}*{Theorem 6.6.3}), the inclusions 
$(M,X)\emb (M,\,X\cup C_U)$ induce an isomorphism $\colim_{U\in\nu}H^n(M,\,X\cup C_U)\simeq H^n(M,X)$.
By map excision (or by usual excision, using that $C_{-U}$ is clopen in $M$),
the retraction map $(M,\,X\cup C_U)\to(C_{-U}\sqcup pt,pt)$ induces an isomorphism
$H^n(C_{-U}\sqcup pt,\,pt)\simeq H^n(M,\,X\cup C_U)$.
Since the composition $(M,X)\to (M,\,X\cup C_U)\to (C_{-U}\sqcup pt,\,pt)$ is the original retraction map,
$\psi^*$ is an isomorphism.
\end{proof}

\begin{lemma} \label{cofinals}
(a) Every neighborhood of $X$ in $M$ contains $C_U$ for some $U\in\nu$.

(b) Every compact subset of $M$ lies in $C_K\cup Q$ for some $K\in\kappa$ and some compact $Q\subset X$.
\end{lemma}

\begin{proof}[Proof. (a)]
Let $N$ be the given neighborhood; we may assume that it is open.
Let $U=\{i\in\N\mid C_i\subset N\}$, then $C_U\subset N$.
Let us embed $M_\odot$ in $M$ by picking one point $x_i$ in each $C_i$, so that if $C_i\not\subset N$, then 
$x_i\notin N$.
Then $U\cup X$ is the preimage of $N$ under this embedding.
Hence $U\cup X$ is open in $M_\odot$, so $U\in\nu$.
\end{proof}

\begin{proof}[(b)]
Let $C$ be the given compact subset and let $K=\{i\in\N\mid C_i\cap C\ne\emptyset\}$.
Then $C\subset C_K\cup Q$, where $Q:=C\cap X$ is compact.
Also $K\cup Q$ is compact since it is the image of $C$ under the quotient map $M\to M_\odot$.
Hence $K$ has compact closure in $M_\odot$, so $K\in\kappa$.
\end{proof}

Let us note that Lemma \ref{cofinals} and \cite{M-I}*{Proposition \ref{fish:duality}} 
have the following consequence:

\begin{corollary} \label{discrete duality}
Let $S\subset\N$.

(a) $S\in\bar\nu$ if and only if $S$ meets every member of $\kappa$ in a finite set.

(b) $S\in\kappa$ if and only if $S$ meets every member of $\bar\nu$ in a finite set. 
\end{corollary}

\subsection{Dual formulation}

The controlled additivity axiom can be dualized.
The dual version asserts that, under the same hypotheses, certain natural maps induce 
for each $n$ an isomorphism
\[H_n(M,X)\simeq\lim_{F\in\bar\nu}\bigoplus_{i\in F}H_n(C_i)=\bigcap_{F\in\bar\nu}
\left(\prod_{i\notin F}H_n(C_i)\oplus\bigoplus_{i\in F}H_n(C_i)\right)\subset\prod_{i\in\N}H_n(C_i),\]
respectively,
\[H^n(M,X)\simeq\lim_{K\in\kappa}\bigoplus_{i\in K}H^n(C_i)=\bigcap_{K\in\kappa}
\left(\prod_{i\notin K}H^n(C_i)\oplus\bigoplus_{i\in K}H^n(C_i)\right)\subset\prod_{i\in\N}H^n(C_i).\]

Namely, the retractions $(M,X)\to(C_F\sqcup pt,pt)$, where $F\in\bar\nu$, induce a map
\[H_n(M,X)\xr{\psi_*}\lim_{F\in\bar\nu} H_n(C_F).\]
Also, the inclusions $C_i\emb C_F$, where $i\in F$, induce a map
\[\bigoplus_{i\in F}H_n(C_i)\xr{\psi_*^F}H_n(C_F).\]
The maps $\psi_*$ and $\psi_*^F$, $F\in\bar\nu$, are required to be isomorphisms.

Dually, the inclusions $(\bar C_K,\,\bar C_K\cap X)\emb\big(M,X)$, where $K\in\kappa$, induce a map
\[H^n(M,X)\xr{\phi^*}\lim_{K\in\kappa} H^n(\bar C_K,\,\bar C_K\cap X).\]
Also, the retractions $(\bar C_K,\,\bar C_K\cap X)\to(C_i\sqcup pt,pt)$, where $i\in K$, induce a map
\[\bigoplus_{i\in K}H^n(C_i)\xr{\phi^*_K}H^n(\bar C_K,\,\bar C_K\cap X).\]
The maps $\phi^*$ and $\phi^*_K$, $K\in\kappa$, are required to be isomorphisms.

\begin{remark} The controlled additivity axiom consists of two requirements: an equivalent form 
(modulo map excision) of $\prod$-additivity and a strengthened form of $\bigoplus$-additivity.
The dualized controlled additivity axiom also consists of two requirements: an equivalent form 
(modulo map excision) of $\bigoplus$-additivity and a strengthened form of $\prod$-additivity.
But in fact, there is no asymmetry, due the following
\end{remark}

\begin{proposition} \label{dualized}
Dualized controlled additivity is equivalent to controlled additivity,
modulo Milnor's three axioms.
\end{proposition}

\begin{proof} The diagrams
\[\begin{CD}
\colim\limits_{K\in\kappa} H_n(\bar C_K,\,\bar C_K\cap X)@>\phi_*>>H_n(M,X)@>\psi_*>>\lim\limits_{F\in\bar\nu} H_n(C_F)\\
@VV{\colim\phi^K_*}V@.@A{\lim\psi_*^F}AA\\
\colim\limits_{K\in\kappa}\prod\limits_{i\in K}H_n(C_i)@.\xr{\qquad\chi_*\qquad}@.
\lim\limits_{F\in\bar\nu}\bigoplus\limits_{i\in F}H_n(C_i)
\end{CD}\]
and
\[\begin{CD}
\colim\limits_{F\in\bar\nu} H^n(C_F)@>\psi^*>> H^n(M,X)@>\phi^*>>\lim\limits_{K\in\kappa} H^n(\bar C_K,\,\bar C_K\cap X)\\
@VV{\colim\psi^*_F}V@.@A{\lim\phi^*_K}AA\\
\colim\limits_{F\in\bar\nu}\prod\limits_{i\in F}H^n(C_i)@.\xr{\qquad\chi^*\qquad}@.
\lim\limits_{K\in\kappa} \bigoplus\limits_{i\in K}H^n(C_i).
\end{CD}\] 
are commutative (we leave the verification to the reader).

The four vertical arrows are isomorphisms by Milnor's axioms.
Of the two arrows in each top line, one is assumed to be an isomorphism and the other needs to be proved
to be an isomorphism.
Thus it remains to show that the natural maps $\chi_*$ and $\chi^*$ in the bottom line are isomorphisms. 

The {\it support} of an element $(g_i)\in\prod_{i\in\N}G_i$, where the $G_i$ are abelian, is
$\{i\in\N\mid g_i\ne0\}$.
If regarded as a map between subgroups of $\prod_{i\in\N}H_n(C_i)$ (resp.\ $\prod_{i\in\N}H^n(C_i)$),
$\chi_*$ (resp.\ $\chi^*$) is the inclusion map (we leave the verification to the reader).

Now, $\colim_{K\in\kappa}\prod_{i\in K}H_n(C_i)$ is the subgroup of $\prod_{i\in\N}H_n(C_i)$
consisting of all $(g_i)$ whose support lies in some $K\in\kappa$.
Also, $\lim_{F\in\bar\nu}\bigoplus_{i\in F}H_n(C_i)$ is the subgroup of
$\prod_{i\in\N}H_n(C_i)$ consisting of all $(g_i)$ whose support meets the every $F\in\bar\nu$ in a finite set.
By Corollary \ref{discrete duality}(b), these two subgroups coincide.

Similarly, $\colim_{F\in\bar\nu}\prod_{i\in F}H^n(C_i)$ is the subgroup of $\prod_{i\in\N}H^n(C_i)$
consisting of all $(g_i)$ whose support lies in some $F\in\bar\nu$.
Also, $\lim_{K\in\kappa} \bigoplus_{i\in K}H^n(C_i)$ is the subgroup of
$\prod_{i\in\N}H^n(C_i)$ consisting of all $(g_i)$ whose support meets every $K\in\kappa$
in a finite set.
By Corollary \ref{discrete duality}(a), these two subgroups coincide.
\end{proof}

\subsection{Examples} 

\begin{example}\label{10.3} Let $\tilde S^n$ denote the $n$-sphere if $n>0$ and a single point (and not two points) 
if $n=0$. 
(This ``reduced $n$-sphere'' is a Moore space $M_n(\Z)$ and also a co-Moore space $M^n(\Z)$.)
Let $S=\tilde S^n\x\N$, where $\N$ denotes the countable discrete space.
If $Y$ is a locally compact separable metrizable space, let $Y^+=Y\cup\{\infty\}$ denote its one-point 
compactification.
The set $\N^\N$ of all functions $\N\to\N$ is partially ordered by $f\le g$ iff $f(n)\le g(n)$ for each $n\in\N$.
We will denote by $\Z_{ij}$ merely a copy of $\Z$.

(a) Let $M=S^+\x\N$ and $X=\{\infty\}\x\N$.
The connected components of $M\but X$ are naturally indexed by $\N\x\N$, and $\kappa$ contains a countable cofinal 
subset $\kappa_\N$ consisting of the products $K_j:=\N\x\{1,\dots,j\}$, $j\in\N$.
Also, $\bar\nu$ contains a cofinal subset $\bar\nu_0$ consisting of the regions $F_f:=\{(i,j)\mid i\le f(j)\}$, $f\in\N^\N$.
See Figure \ref{kappa-nu}(a).

The composite isomorphism 
\[\colim\limits_{K\in\kappa_\N}\prod\limits_{i\in K}H_n(C_i)\simeq
\colim\limits_{K\in\kappa}\prod\limits_{i\in K}H_n(C_i)\simeq H_n(M,X)\simeq
\lim\limits_{F\in\bar\nu}\bigoplus\limits_{i\in F}H_n(C_i)\simeq
\lim\limits_{F\in\bar\nu_0}\bigoplus\limits_{i\in F}H_n(C_i)\]
takes the form
\[\bigoplus_{j\in\N}\prod_{i\in\N}\Z_{ij}\simeq\lim_{f\in\N^\N}\bigoplus_{j\in\N}\bigoplus_{i\le f(j)}\Z=
\bigcap_{f\in\N^\N}\Big(\prod_{j\in\N}\prod_{i>f(j)}\Z_{ij}
\oplus\bigoplus_{j\in\N}\bigoplus_{i\le f(j)}\Z_{ij}\Big)\subset\prod_{(i,j)\in\N^2}\Z_{ij},\tag{$*$}\]
where the bonding maps $\bigoplus_j\bigoplus_{i\le f(j)}\Z\to\bigoplus_j\bigoplus_{i\le g(j)}\Z$, $f\ge g$,
are given by the coordinatewise projections $\bigoplus_{i\le f(j)}\Z\to\bigoplus_{i\le g(j)}\Z$ 
onto the first $g(j)$ factors.
One can rewrite ($*$) as
\[\big(\Z[[x]]\big)[y]\simeq\lim_{f\in\N^\N}\bigoplus_{n=0}^\infty y^n\Z[x]\big/\big<x^{f(n)}\big>=
\bigcap_{f\in\N^\N}\Big(\Z[x,y]+\prod_{n=0}^\infty x^{f(n)}y^n\Z[[x]]\Big)
\subset\Z[[x,y]].\]

The composite isomorphism
\[\lim\limits_{K\in\kappa_\N}\bigoplus\limits_{i\in K}H^n(C_i)\simeq
\lim\limits_{K\in\kappa} \bigoplus\limits_{i\in K}H^n(C_i)\simeq H^n(M,X)\simeq
\colim\limits_{F\in\bar\nu}\prod\limits_{i\in F}H^n(C_i)\simeq
\colim\limits_{F\in\bar\nu_0}\prod\limits_{i\in F}H^n(C_i)\]
takes the form
\[\prod_{j\in\N}\bigoplus_{i\in\N}\Z_{ij}\simeq\colim_{f\in\N^\N}\prod_{j\in\N}\prod_{i\le f(j)}\Z=
\bigcup_{f\in\N^\N}\prod_{j\in\N}\prod_{i\le f(j)}\Z_{ij}\subset\prod_{(i,j)\in\N^2}\Z_{ij},\tag{$**$}\]
where the bonding maps $\prod_j\prod_{i\le g(j)}\Z\to\prod_j\prod_{i\le f(j)}\Z$, $g\le f$,
are given by the coordinatewise inclusions $\prod_{i\le g(j)}\Z\to\prod_{i\le f(j)}\Z$ onto the first $g(j)$ 
factors.
One can rewrite ($**$) as
\[\big(\Z[x]\big)[[y]]\simeq
\bigcup_{f\in\N^\N}\Big\{\sum_{n=0}^\infty P_ny^n\ \Big|\ P_n\in\Z[x],\,\deg P_n\le f(n)\Big\}
\subset\Z[[x,y]].\]

\begin{figure}[h]
\includegraphics[width=15cm]{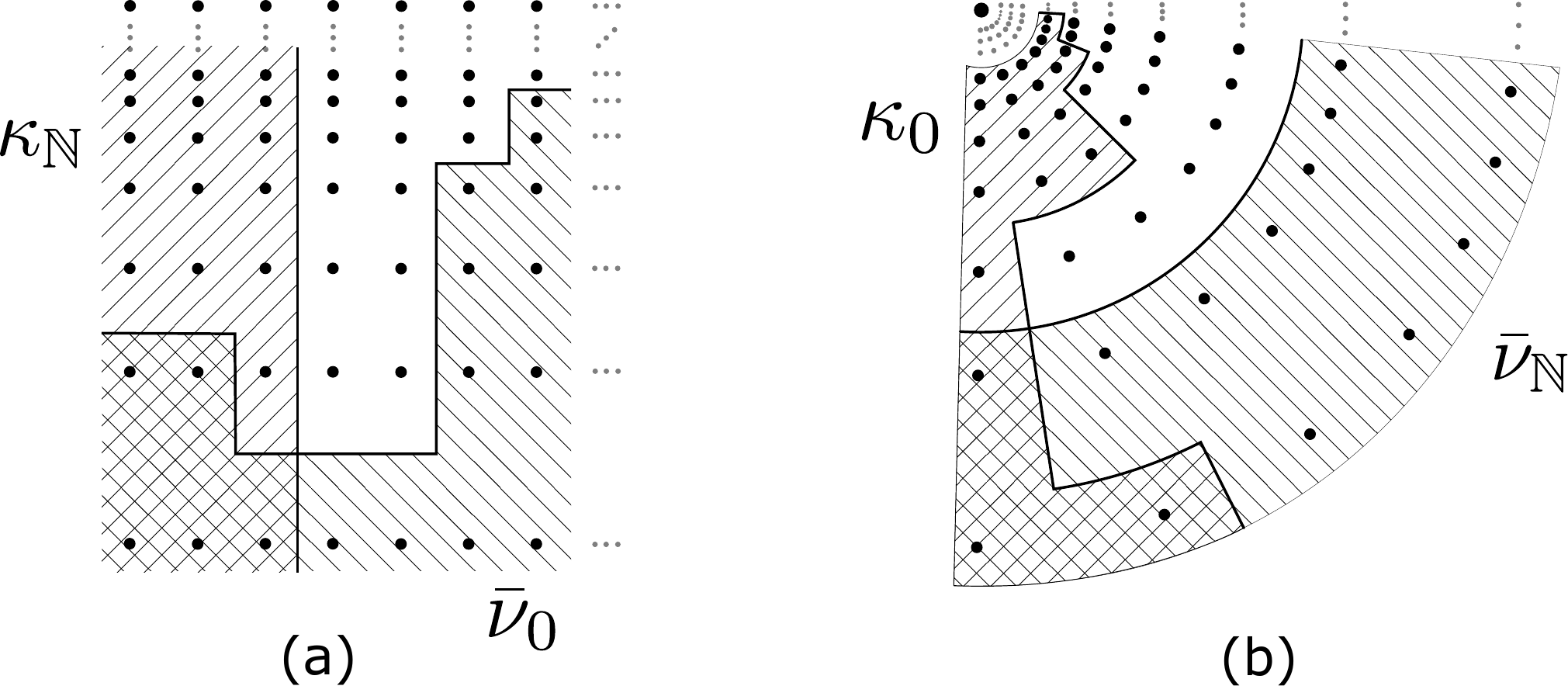}
\caption{Cofinal subsets of $\kappa$ and $\bar\nu$ for $(M,X)$ and $(M/X,\{X\})$ where
$M=\mathbb N\times\mathbb N^+$ and $X=\mathbb N\times\{\infty\}$}
\label{kappa-nu}
\end{figure}

(b) Now let $M$ be the metric quotient $(S^+\x\N)/(\{\infty\}\x\N)$, which has the topology of 
the quotient uniformity (and not the quotient topology, which is non-metrizable; see 
\cite{M2}*{\S\ref{metr:quotients}} or \cite{M-III}*{\S\ref{lim:metric quotient}} for the details).
For example, in the case $n=0$ it is homeomorphic to the subset $\{(0,0)\}\cup (C\cap R)$ of the square 
$[0,1]\x[0,1]\subset\R^2$, where $C$ is the the union of all circles $x^2+y^2=\frac1n$, $n\in\N$, and 
$R$ is the union of all radial lines $y=kx$, $k\in\N$.
Let $X$ be the point $\{\{\infty\}\x\N\}$ of $M$; let us note that $M$ is not locally compact at this point. 

The connected components of $M\but X$ are naturally indexed by $\N\x\N$, and $\bar\nu$ contains a countable 
cofinal subset $\bar\nu_\N$ consisting of the products $F_i:=\{1,\dots,i\}\x\N$, $i\in\N$.
Also, $\kappa$ contains a cofinal subset $\kappa_0$ consisting of the regions $K_f:=\{(i,j)\mid j\le f(i)\}$, 
$f\in\N^\N$. 
See Figure \ref{kappa-nu}(b).

The composite isomorphism 
\[\lim\limits_{F\in\bar\nu_\N}\bigoplus\limits_{i\in F}H_n(C_i)\simeq
\lim\limits_{F\in\bar\nu}\bigoplus\limits_{i\in F}H_n(C_i)\simeq H_n(M,X)\simeq
\colim\limits_{K\in\kappa}\prod\limits_{i\in K}H_n(C_i)\simeq
\colim\limits_{K\in\kappa_0}\prod\limits_{i\in K}H_n(C_i)\]
takes the form ($**$).
The composite isomorphism
\[\colim\limits_{F\in\bar\nu_\N}\prod\limits_{i\in F}H^n(C_i)\simeq
\colim\limits_{F\in\bar\nu}\prod\limits_{i\in F}H^n(C_i)\simeq H^n(M,X)\simeq
\lim\limits_{K\in\kappa} \bigoplus\limits_{i\in K}H^n(C_i)\simeq
\lim\limits_{K\in\kappa_0} \bigoplus\limits_{i\in K}H^n(C_i)\]
takes the form ($*$).
\end{example}

\begin{example} \label{non-closed}
Let $(M,X)$ be as in Example \ref{10.3}(a).
The inclusion $j\:M\but X\to M$ is not a closed map since $M\but X$ is not closed in $M$.
The projection $p\:M\to M/X$ onto the metric quotient (as in Example \ref{10.3}(b)) is also not a closed map, 
since $\{(n,n)\mid n\in\N\}$, regarded as a subset of $\N\x\N=M\but X$, is closed in $M$.
The continuous bijection $b\:M/X\to(M\but X)^+$ with the one-point compactification is again not a closed map,
since $\{(1,n)\mid n\in\N\}$ is closed in $M/X$.
Each of $j$, $p$ and $b$ induces a non-surjective injection in homology and in cohomology:
\[\bigoplus_{i\in\N\x\N}\Z\xrightarrow[b^*]{j_*}\bigoplus_{m\in\N}\prod_{n\in\N}\Z\xrightarrow[p^*]{p_*}
\prod_{m\in\N}\bigoplus_{n\in\N}\Z\xrightarrow[j^*]{b_*}\prod_{i\in\N\x\N}\Z.\]
\end{example}

\section{Uniqueness theorem} \label{uniqueness-chapter}

We are now ready to state the main result of this paper.

\begin{theorem} \label{uniqueness}
If $H$ and $H'$ are ordinary (co)homology theories on the category $C_{\nu\kappa}$ of 
closed pairs of Polish spaces satisfying map excision and controlled additivity, then any graded isomorphism 
$H(pt)\to H'(pt)$ extends to a natural equivalence from $H$ to $H'$.
\end{theorem}

\subsection{Sketch of proof} \label{uniqueness-section}

The strategy of proof of Theorem \ref{uniqueness} is to generalize Milnor's proof for $\C_\kappa$ \cite{Mi1},
except for one step (the vanishing of infinite-dimensional phantoms), which does not seem to generalize 
(because of the lack of a proof of \cite{M-IV}*{Theorem \ref{pos:homology spectral seq}} directly from the axioms), 
but luckily can be circumvented by drawing from Milnor's proof for $\C_\nu$ \cite{Mi2}.

Let $X$ be a Polish space.
By \cite{M-I}*{Theorem \ref{fish:isbell}}, $X$ is the limit of an inverse sequence of polyhedra and PL maps
$\dots\xr{p_1}P_1\xr{p_0}P_0$.
(They can be assumed to be locally compact and finite dimensional, but we will not need this.)
Moreover, we may assume that the $P_i$ come endowed either with triangulations $L_i$ so that each $p_i$ 
is simplicial as a map from $L_{i+1}$ to a certain subdivision of $L_i$ (see 
\cite{M-I}*{Remark \ref{fish:nerves and cubes}(3,5)}) 
or with cubulations $L_i$ so that each $p_i$ is cubical 
(that is, its restriction to each cube is the projection onto a subcube) as a map from $L_{i+1}$ to 
the standard cubical subdivision $L_i^\#$ of $L_i$ (see \cite{M-I}*{Remark \ref{fish:nerves and cubes}(1)}).
Using the product cell structure of $L_i\x [i-1,\,i]$, we make the infinite mapping telescope
$P_{[0,\infty)}$ into a cell complex $L$, that is, a CW complex whose attaching maps are PL embeddings.
(We could of course subdivide this cell complex into a simplicial complex, but we will not need this.)

We may assume for simplicity (without loss of generality) that $P_0=pt$, and hence 
$P_{[0,\infty]}=P_{[0,\infty)}\cup X$ is contractible (see Proposition \ref{fine telescope} concerning 
$P_{[0,\infty]}$).
Then the reduced (co)homology of $X$ is determined by that of the pair $(P_{[0,\infty]},\,X)$.
Our principal goal is to show, using only the axioms, that the latter can be computed from a combinatorial (co)chain 
complex determined by the cellulation of $P_{[0,\infty)}$ and hence does not depend on the choice of the theory.

It will be convenient to describe these chain and cochain complexes in a slightly more general setting.

\subsection{Filtrations of subcomplexes} \label{subcomplex-filtrations}

Suppose that $(M,X)$ is a closed pair of Polish spaces such that $M\but X$ is cellulated by a cell complex $L$ 
such that the cells of $L$ are scattered towards $X$.
Let 
\begin{alignat*}{1}
\kappa=\{&\text{ subcomplexes $C$ of $L$ such that the closure of $C$ in $M$ is compact }\},\\
\nu=\{&\text{ subcomplexes $N$ of $L$ such that $N\cup X$ is a neighborhood of $X$ in $M$ }\},\\
\bar\nu=\{&\text{ subcomplexes $Q$ of $L$ such that $Q$ is a closed subset of $M$ }\}.
\end{alignat*}

\begin{lemma} \label{simplicial cofinality} Let $S$ be a closed subset of $L$.

(a) If $S$ is closed in $M$, then it lies in a member of $\bar\nu$.

(a$'$) If $S$ meets every member of $\bar\nu$ in a compact set, then it meets every subset of $L$ 
that is closed in $M$ in a compact set.

(b$'$) If $S$ meets every member of $\kappa$ in a compact set, then it meets every compact subset of $M$ 
in a compact set.
\end{lemma}

Part (a$'$) follows trivially from (a). 
Part (b$'$) would similarly follow from its ``dual'':

{\it Assertion (b).} ``Every compact subset of $M$ lies in $K\cup Q$ for some $K\in\kappa$ and 
some compact $Q\subset X$.''

However, Assertion (b) is false in general, since $L$ need not be locally finite.
But it is true ``up to homotopy'' (see Proposition \ref{homotopy inclusion}).

\begin{proof}[Proof. (a)] 
Let $N$ be the union of all closed cells of $L$ that meet $S$.
Suppose that a sequence of points $x_i\in N$ converges to an $x\in X$.
Then each $x_i$ lies in the same closed cell of $L$ with some $y_i\notin S$.
But then the $y_i$ also converge to $x$.
This contradicts our hypothesis that $S$ is closed in $M$ and lies in $L$.
Hence $N$ is closed in $M$.
\end{proof}

\begin{proof}[(b$'$)] Suppose that $K$ is a compact subset of $M$ such that $K\cap S$ is non-compact.
Then there exists a sequence of points $x_i\in K\cap S$ that converges to a point $x\in K\cap X$.
Let $\sigma_i$ be the minimal closed cell of $L$ containing $x_i$, and let $C=\bigcup_{i\in\N}\sigma_i$.
Then the closure of $C$ in $M$ is $C\cup\{x\}$ and it is compact. 
Thus the intersection of $S$ with the member $C$ of $\kappa$ is non-compact, since it contains 
the $x_i$ but does not contain their limit $x$.
\end{proof}

\begin{corollary} \label{simplicial duality}
Let $S$ be a subcomplex of $L$.

(a) $S$ lies in a member of $\bar\nu$ if and only if every member of $\kappa$ meets $S$ in a finite subcomplex
of $L$.

(b) $S$ lies in a member of $\kappa$ if and only if every member of $\bar\nu$ meets $S$ in a finite subcomplex
of $L$.
\end{corollary}

\begin{proof} The ``only if'' assertions follow from those of \cite{M-I}*{Proposition \ref{fish:duality}}.
The ``if'' assertions follow from those of \cite{M-I}*{Proposition \ref{fish:duality}} and from parts 
(a$'$), (b$'$) of Lemma \ref{simplicial cofinality}.
\end{proof}

The following ``correction'' of Assertion (b) above is not needed for the proof of 
the Uniqueness Theorem (Theorem \ref{uniqueness}).

\begin{proposition} \label{homotopy inclusion}
Suppose that $L$ is a simplicial complex.
Then every compact subset of $M$ lies in $U\cup Q$, where $Q$ is a compact subset of $X$ and $U$ is an open subset 
of $L$ such that $U\cup Q$ deformation retracts onto $C\cup Q$ keeping $Q$ fixed, where $C$ is 
a member of $\kappa$.
\end{proposition}

\begin{proof} Let $K$ be the given compact subset of $M$, and let $Q=K\cap X$.
Since $K\but X$ is locally compact, it can be represented as a union of compacta $K_1\subset K_2\subset\dots$ 
such that each $K_i\subset\Int_K K_{i+1}$.
Let $K_0=\emptyset$.
Let $(U_i)_{i\in L_0}$ be the cover of $L$ by the open stars of its vertices.
For each $n=1,2,\dots$ there exists a finite $V_n\subset L_0$ such that $\bigcup_{i\in V_n} U_i$ contains 
$K_n\but\Int_K K_{n-1}$.
We may assume that each $U_i$ with $i\in V_n$ actually meets $K_n\but\Int_K K_{n-1}$.
Let $V=\bigcup_{n\in\N}V_n$.
Let $C$ be the subcomplex of $L$ consisting of all simplexes with vertices in $V$.
Then $U:=\bigcup_{i\in V} U_i$ deformation retracts onto $C$.
Clearly, this deformation retraction extends to a deformation retraction of $U\cup Q$ onto $C\cup Q$ 
keeping $Q$ fixed.

It remains to prove that the closure of $C$ is compact.
Let $W_n=V_1\cup\dots\cup V_n$ and let $C_n$ be the subcomplex of $L$ consisting of all simplexes with 
vertices in $W_n$.
Then each $C_i\subset C_{i+1}$.
Each simplex of $C$ has a finite number of vertices, so it must be in some $C_n$.
Thus $C=\bigcup_{n\in\N} C_n$.

Let $(x_i)$ be a sequence of points in $C$.
If all of them are in some $C_n$, then there is a convergent subsequence.
Else there is a subsequence of points $x_{k_i}\in C\but C_i$.
Then each $x_{k_i}$ lies in a simplex of $C$ which has a vertex $v_i\in V\but W_i$.
Then $v_i\in V_{n_i}$ for some $n_i>i$ and hence the open star $U_{v_i}$ of $v_i$
contains a point $y_i\in K_{n_i}\but\Int_K K_{n_i-1}$.
In particular, $y_i\in K\but\Int_K K_i$.
Then $(y_i)$ converges to some $x\in Q$.
Since the simplexes of $L$ are scattered towards $X$, $(v_i)$ also converges to $x$, and consequently also 
$(x_{k_i})$ converges to $x$.
Thus $C$ is pre-compact.
\end{proof}

\subsection{Chain complexes} \label{chain complexes}

Let $P$ be any subcomplex of $L$, and let $P_n=P^{(n)}\but P^{(n-1)}$ be the set of its $n$-cells.
If $\sigma\in P_n$, let $G_\sigma$ be a copy of the coefficients $G$ of the ordinary theory, 
i.e.\ $G_\sigma=H_n(\sigma,\partial\sigma)$ (resp.\ $H^n(\sigma,\partial\sigma)$).
For an element $g=(g_\sigma)\in\prod_{\sigma\in P_n}G_\sigma$ let $\supp(g)=\{\sigma\in P_n\mid g_\sigma\ne 0\}$.
Let us define four subgroups of $\prod_{\sigma\in P_n}G_\sigma$, each containing $\bigoplus_{\sigma\in P_n}G_\sigma$:
\begin{alignat*}{1}
C_n^\kappa(P)=\Big\{g\in\prod_{\sigma\in P_n}G_\sigma\ \Big|\ &
\supp(g)\subset K\text{ for some }K\in\kappa\Big\}\\
C^n_\kappa(P)=\Big\{g\in\prod_{\sigma\in P_n}G_\sigma\ \Big|\ &
\supp(g)\cap K\text{ is finite for each }K\in\kappa\Big\}\\
C_n^\nu(P)=\Big\{g\in\prod_{\sigma\in P_n}G_\sigma\ \Big|\ &
\supp(g)\cap F\text{ is finite for each }F\in\bar\nu\Big\}\\
C^n_\nu(P)=\Big\{g\in\prod_{\sigma\in P_n}G_\sigma\ \Big|\ &
\supp(g)\subset F\text{ for some }F\in\bar\nu\Big\}
\end{alignat*}
Corollary \ref{simplicial duality} guarantees that 
\[C_n^\kappa(P)=C_n^\nu(P)\qquad\text{and}\qquad C^n_\kappa(P)=C^n_\nu(P).\]
Let us note that
\begin{alignat*}{3}
C_n^\kappa(P)&\ \simeq\ \colim_{K\in\kappa}&\prod_{\sigma\in P_n\cap K}G_\sigma&\ \simeq\ 
\colim_{K\in\kappa}\ \lim_{N\in\nu}\prod_{\sigma\in (P_n\cap K)\but N}G_\sigma\\
C^n_\kappa(P)&\ \simeq\ \lim_{K\in\kappa}&\bigoplus_{\sigma\in P_n\cap K}G_\sigma&\ \simeq\ 
\lim_{K\in\kappa}\ \colim_{N\in\nu}\prod_{\sigma\in (P_n\cap K)\but N}G_\sigma\\
C_n^\nu(P)&\ \simeq\ \lim_{N\in\nu}&\bigoplus_{\sigma\in P_n\but N}G_\sigma&\ \simeq\ 
\lim_{N\in\nu}\ \colim_{K\in\kappa}\prod_{\sigma\in (P_n\cap K)\but N}G_\sigma\\
C^n_\nu(P)&\ \simeq\ \colim_{N\in\nu}&\prod_{\sigma\in P_n\but N}G_\sigma&\ \simeq\ 
\colim_{N\in\nu}\ \lim_{K\in\kappa}\prod_{\sigma\in (P_n\cap K)\but N}G_\sigma
\end{alignat*}
Of course, the finite product $G_{KN}:=\prod_{\sigma\in (P_n\cap K)\but N}G_\sigma$ is isomorphic to the usual 
relative cellular chain and cochain groups $C_n(P\cap K,\,P\cap K\cap N)$ and $C^n(P\cap K,\,P\cap K\cap N)$ with 
coefficients in $G$.
It follows from the universal properties of (co)limit that the boundary and coboundary maps of
the corresponding chain and cochain complexes admit a unique natural extension to the (co)limit chain and 
cochain groups.
This can be used to form (co)chain complexes out of the new (co)chain groups: 
\begin{alignat*}{4}
C_*^\kappa(P)&\ \simeq\ \colim_{K\in\kappa}&C_*^\infty(P\cap K)&\ \simeq\ 
\colim_{K\in\kappa}\ \lim_{N\in\nu}\,C_*(P\cap K,\,P\cap K\cap N)\\
C^*_\kappa(P)&\ \simeq\ \lim_{K\in\kappa}&C^*_c(P\cap K)&\ \simeq\ 
\lim_{K\in\kappa}\ \colim_{N\in\nu}\,C^*(P\cap K,\,P\cap K\cap N)\\
C_*^\nu(P)&\ \simeq\ \lim_{N\in\nu}&C_*(P,\,P\cap N)&\ \simeq\ 
\lim_{N\in\nu}\ \colim_{K\in\kappa}\,C_*(P\cap K,\,P\cap K\cap N)\\
C^*_\nu(P)&\ \simeq\ \colim_{N\in\nu}&\ C^*(P,\,P\cap N)&\ \simeq\ 
\colim_{N\in\nu}\ \lim_{K\in\kappa}\,C^*(P\cap K,\,P\cap K\cap N)
\end{alignat*}
Here $C^*_c(P\cap K)\simeq C^*_c(P\cap K;\,\Z)\otimes G$ is the usual complex of cellular cochains with 
compact support, and $C_*^\infty(P\cap K)\simeq\Hom\big(C^*_c(P\cap K;\,\Z);\,G\big)$ 
is the usual complex of possibly infinite cellular chains.

The natural maps \[\colim_{K\in\kappa}\ \lim_{N\in\nu}\,G_{KN}\xr{\phi}\lim_{N\in\nu}\ \colim_{K\in\kappa}\,G_{KN},\]
\[\colim_{N\in\nu}\ \lim_{K\in\kappa}\,G_{KN}\xr{\psi}\lim_{K\in\kappa}\ \colim_{N\in\nu}\,G_{KN}\]
are easily seen to commute with the thus extended boundary and coboundary maps. 
Since $\phi$ and $\psi$ are actually isomorphisms, we conclude that
\[C_*^\kappa(P)\simeq C_*^\nu(P)\qquad\text{and}\qquad C^*_\kappa(P)\simeq C^*_\nu(P)\]
also as (co)chain complexes.
So we will also denote these (co)chain complexes by $C_*^X(P)$ and $C^*_X(P)$ when their specific construction
does not concern us.
We also have the relative (co)chain complexes $C_*^X(P,Q):=C_*^X(P)/C_*^X(Q)$ and $C^*_X(P,Q):=C^*_X(P)/C^*_X(Q)$.

\begin{remark} It is easy to see that $\lim\Hom(G_\alpha,G)\simeq\Hom(\colim G_\alpha,\,G)$ for a direct system 
of abelian groups $G_\alpha$.
For $G=\Z$, we also have $\Hom\big(\prod_{\N}\Z,\,\Z\big)\simeq\bigoplus_{\N}\Z$, the isomorphism being
provided by the inclusion $\bigoplus_{\N}\Z\to\Hom\Big(\Hom\big(\bigoplus_{\N}\Z,\,\Z\big),\,\Z\Big)$
(see \cite{Fu}*{Corollary 94.6}).
Hence $C^n_\kappa\simeq\Hom(C_n^\kappa,\Z)$ and $C_n^\nu\simeq\Hom(C^n_\nu,\Z)$, and it follows that
$C^*_X(P)\simeq\Hom\big(C_*^X(P),\,\Z\big)$ and $C_*^X(P)\simeq\Hom\big(C^*_X(P),\,\Z\big)$ (only for chain
and cochain complexes with integer coefficients).
\end{remark}

\subsection{Computation of axiomatic (co)homology}

\begin{theorem} \label{computation}
If $H$ is an ordinary (co)homology theory on $\C_{\nu\kappa}$ satisfying map excision and 
controlled additivity, then
\[H_n(M,X)\simeq H_n\big(C_*^X(L)\big),\]
respectively,
\[H^n(M,X)\simeq H^n\big(C^*_X(L)\big).\]

Moreover, if $M'$ is a closed subset of $M$ such that $L':=M\cap L$ is a subcomplex of $L$, then
\[H_n(M,\,X\cup M')\simeq H_n\big(C_*^X(L,L')\big),\]
respectively,
\[H^n(M,\,X\cup M')\simeq H^n\big(C^*_X(L,L')\big).\]
\end{theorem}

We will now prove Theorem \ref{computation}.
We explicitly discuss only homology except when the case of cohomology is not completely similar.
Also, we explicitly discuss only the case $M'=\emptyset$.

Let $Q^n$ be the closed subset of $M$ consisting of $X$ and of one $n$-ball $\frac12\sigma$ in the interior of each 
$n$-cell $\sigma\in L_n$.
Let $R^n$ be the closed subset of $M$ consisting of $X$ and of the boundaries $\partial(\frac12\sigma)$ for all
$\sigma\in L_n$.  
Let $r_n\:Q^n\to M$ fix $X$ and stretch each $\frac12\sigma$ homeomorphically onto $\sigma$. 
Clearly, $r_n$ is continuous and closed.
Then from the controlled additivity axiom applied to $Q$ and to $R$ and from the map excision axiom applied to
$r_n\:(Q^n,R^n)\to(L^{(n)}\cup X,\,L^{(n-1)}\cup X)$ we obtain the following assertion (a):

\begin{lemma} \label{boundary map}
(a) $H_n(L^{(n)}\cup X,\,L^{(n-1)}\cup X)\simeq C_n^X(L)$ and
$H_i(L^{(n)}\cup X,\,L^{(n-1)}\cup X)=0$ for $i\ne n$.

(b) Under the isomorphism of (a), the boundary map 
\[H_n(L^{(n)}\cup X,\,L^{(n-1)}\cup X)\to H_{n-1}(L^{(n-1)}\cup X,\,L^{(n-2)}\cup X)\] 
from the exact sequence of the triple $(L^{(n)}\cup X,\,L^{(n-1)}\cup X,\,L^{(n-2)}\cup X)$ corresponds to
the boundary map of the chain complex $C_*^X(L)$.
\end{lemma}

\begin{proof}[Proof of (b)] If $L$ is a finite complex, then $M=X\sqcup L$, so (b) reduces to its special case 
$X=\emptyset$.
Also, $C^X_*(L)$ is the usual cellular chain complex $C_*(L)$.
Then the assertion holds for any homology theory (see \cite{ES}).

In the general case, both boundary maps are determined by the corresponding boundary maps on finite subcomplexes
using the universal properties of limit and colimit.
\end{proof}

\begin{lemma} \label{chains} 
(a) $H_i(L^{(n)}\cup X,\,X)$ is zero for $i>n$ and isomorphic to $H_i\big(C_*^X(L)\big)$ for $i<n$.
Also $H_n(L^{(n)}\cup X,\,X)\simeq Z_n\big(C_*^X(L)\big)$.

(b) Under the latter isomorphism and the isomorphism of Lemma \ref{boundary map}(a), the maps
\[H_{n+1}(L^{(n+1)}\cup X,\,L^{(n)}\cup X)\xr{\partial_n} H_n(L^{(n)}\cup X,\,X)\xr{j_*} 
H_n(L^{(n)}\cup X,\,L^{(n-1)}\cup X)\]
correspond to the maps
\[C_{n+1}^X(L)\xr{\partial} Z_n\big(C_*^X(L)\big)\xr{i} C_n^X(L).\]
\end{lemma}

\begin{proof} The case $n=0$ follows from Lemma \ref{boundary map}, using that $i$ and $j_*$ are isomorphisms
for $n=0$.
Let us assume that the lemma is known for $n-1$ and prove it for the given $n$.
We consider the exact sequence of the triple $(L^{(n)}\cup X,\,L^{(n-1)}\cup X,\,X)$ and use Lemma 
\ref{boundary map}(a).
The assertion for $i\ne n,n-1$ is immediate, and in the remaining cases we have
\[0\to H_n(L^{(n)}\cup X,\,X)\to C_n^X(L)\xr{\partial}
Z_{n-1}\big(C_*^X(L)\big)\to H_{n-1}(L^{(n)}\cup X,\,X)\to 0.\]
This implies the assertion for $i=n-1$ and also that $j_*\simeq i$.
But since $i$ is an inclusion and $j_*\partial_n\simeq i\partial$ by Lemma \ref{boundary map}(b), we conclude that  
$\partial_n\simeq\partial$.
\end{proof}

\begin{lemma} \label{skeletal phantoms}
(a) $H_n(M,X)\simeq\colim_{i\in\N} H_n(L^{(i)}\cup X,\,X)$.

(b) There is a short exact sequence 
\[0\to\lim_{i\in\N} H^n(L^{(i)}\cup X,\,X)\to H^n(M,X)\to \derlim_{i\in\N}H^{n-1}(L^{(i)}\cup X,\,X)\to 0.\]
\end{lemma}

\begin{proof} Clearly, $(M,X)$ is homotopy equivalent to $\big(M\x[0,\infty),\,X\x[0,\infty)\big)$.
Let $L^{[0,\infty)}\subset L\x[0,\infty)$ be the mapping telescope of the inclusions 
$L^{(0)}\subset L^{(1)}\subset\dots$.
It is not hard to construct a deformation retraction of $\big(M\x[0,\infty),\,X\x[0,\infty)\big)$ onto
$\big(L^{[0,\infty)}\cup Z,\,Z\big)$, where $Z=\Fr L^{[0,\infty)}$, which is a closed subset of $X\x[0,\infty)$.
Also, \[\big(L^{[0,\infty)}\cup Z,\,Z\big)\emb
\big(L^{[0,\infty)}\cup X\x[0,\infty),\,X\x[0,\infty)\big)\]
induces isomorphisms on homology and cohomology by map excision.
Finally, by standard arguments (see \cite{Mi2}*{Lemmas 1, 2}), any (co)homology theory satisfying 
$\bigsqcup$-additivity
satisfies \[H_n\big(L^{[0,\infty)}\cup X\x[0,\infty),\,X\x[0,\infty)\big)\simeq\colim_{i\in\N} H_n(L^{(i)}\cup X,\,X)\]
and a short exact sequence
\begin{multline*}
0\to\lim_{i\in\N} H^n(L^{(i)}\cup X,\,X)\to H_n\big(L^{[0,\infty)}\cup X\x[0,\infty),\,X\x[0,\infty)\big)\to \\
\derlim_{i\in\N}H^{n-1}(L^{(i)}\cup X,\,X)\to 0.
\end{multline*}
\end{proof}

\begin{proof}[Proof of Theorem \ref{computation}]
By Lemma \ref{chains}, the inclusion induced maps \[H_n(L^0\cup X,\,X)\to H_n(L^1\cup X,\,X)\to\dots\] are 
eventually isomorphisms, and their stable image $\colim_i H_n(L^i\cup X,\, X)$ is isomorphic to
$H_n\big(C_*^X(L)\big)$.
On the other hand, by Lemma \ref{skeletal phantoms} $\colim_i H_n(L^i\cup X,\, X)\simeq H_n(M,X)$.

The case of cohomology is similar, since the bonding maps in 
\[\dots\to H^{n-1}(L^1\cup X,\,X)\to H^{n-1}(L^0\cup X,\,X)\] are eventually isomorphisms, and so their $\derlim$ 
vanishes.
\end{proof}

\subsection{Proof} \label{uniqueness-proof}

\begin{proof}[Proof of Theorem \ref{uniqueness}]
As in the above strategy of proof, let us consider an extended mapping telescope $P_{[0,\infty]}$ for $X$,
but not necessarily with $P_0=pt$.
We have $H_n(X)\simeq H_n(X\sqcup P_0,\,P_0)$, which is in turn isomorphic to $H_{n+1}(P_{[0,\infty]},\,X\cup P_0)$ 
by considering the triple $(P_{[0,\infty]},\,X\cup P_0,\,P_0)$.
By Theorem \ref{computation} with $(M,M')=(P_{[0,\infty]},P_0)$,
\[H_{n+1}(P_{[0,\infty]},\,X\cup P_0)\simeq H_{n+1}\big(C^X_*(P_{[0,\infty)},P_0)\big).\]
Moreover, this isomorphism is natural with respect to any map $f\:X\to Y$, in the following sense.
If $Q_{[0,\infty]}$ is an extended mapping telescope for $Y$, by Lemma \ref{Milnor} $f$ extends to a 
map $\bar f\:P_{[0,\infty]}\to Q_{[0,\infty]}$ such that $\bar f^{-1}(Y)=X$.
Since $Q_0$ is a deformation retract of $Q_{[0,\infty]}$, we may assume that $\bar f(P_0)\subset Q_0$.
Thus $\bar f$ induces a map
\[H_{n+1}(P_{[0,\infty]},\,X\cup P_0)\to H_{n+1}(Q_{[0,\infty]},\,Y\cup Q_0).\]
Also, it is not hard to see that its restriction $F\:P_{[0,\infty)}\to Q_{[0,\infty)}$ induces a map
\[H_{n+1}\big(C^X_*(P_{[0,\infty)},P_0)\big)\to H_{n+1}\big(C^Y_*(Q_{[0,\infty)},Q_0)\big).\]
Then it follows from the proof of Theorem \ref{computation} that its isomorphisms commute with these two maps.
(Let us only note that we may assume $F$ to be cellular, so that 
$F(P_{[0,\infty]}^{(n)})\subset Q_{[0,\infty]}^{(n)}$ for each $n$.)
 
The said also applies to the other theory $H'$.
Thus we get a commutative diagram
\[\begin{CD}
H_n(X)@.\simeq H_{n+1}\big(C^X_*(P_{[0,\infty)},P_0)\big)@.\simeq H'_n(X)\\
@VV\bar f_*V@VV\bar f_*V@V\bar f_*VV\\
H_n(Y)@.\simeq H_{n+1}\big(C^Y_*(Q_{[0,\infty)},Q_0)\big)@.\simeq H'_n(Y)
\end{CD}\]
which provides the desired natural equivalence on single spaces.
The proof for pairs $(X,X')$ is similar, using Theorem \ref{computation} with 
$(M,M')=(P_{[0,\infty]},\,P'_{[0,\infty]}\cup P_0)$.
The proof for cohomology is also similar.
\end{proof}

\section{Addendum: Telescopic filtrations revisited} \label{chain complexes2}

The proof of the uniqueness theorem provides a method of computation of Steenrod--Sitnikov homology $H_i(X)$
and \v Cech cohomology $H^i(X)$ of the Polish space $X$ in terms of certain cellular (co)chain complexes.
It is natural ask whether these (co)chain complexes or their description can be simplified.

Let $P_{[0,\infty]}=P_{[0,\infty)}\cup X$ be as in \cite{M-I}*{Proposition \ref{fish:fine telescope}}, 
now assuming that
each $P_i$ is locally compact, and let each $P_i$ be cellulated by a cell complex $L_i$ and $P_{[0,\infty)}$ 
be cellulated by a cell complex $L$ as in \S\ref{uniqueness-section}. 
Let $\kappa$ and $\bar\nu$ be defined as in \S\ref{subcomplex-filtrations}.

We call an inverse sequence $\dots\xr{q_1}Q_1\xr{q_0}Q_0$ a {\it restriction} of the inverse sequence
$\dots\xr{p_1}P_1\xr{p_0}P_0$ if each $Q_i$ is a subcomplex of $L_i$ and each $q_i=p_i|_{Q_{i+1}}$ 
(in particular, each $p_i(Q_{i+1})\i Q_i$).

Let $\kappa'$ consist of the mapping telescopes $K_{[0,\infty)}$ of all restrictions $\dots\to K_1\to K_0$
of $\dots\to P_1\to P_0$ such that each $K_i$ is a finite subcomplex of $L_i$.

Let $\bar\nu'$ consist of the mapping telescopes $N_{[0,\infty)}$ of all restrictions $\dots\to N_1\to N_0$ of $\dots\to P_1\to P_0$ such that $\lim N_i=\emptyset$.

\begin{proposition} \label{telescope filtration}
(a) $\kappa'$ is a cofinal subset of $\kappa$.

(b) $\bar\nu'$ is a cofinal subset of $\bar\nu$.
\end{proposition}

\begin{proof}[Proof. (a)] The closure $K_{[0,\infty]}$ of every member $K_{[0,\infty)}$ of $\kappa'$ 
in $P_{[0,\infty]}$ is compact, since it is the inverse limit of the finite telescopes $K_{[0,i]}$.
Thus $\kappa'\subset\kappa$.

Let $C$ be a subcomplex of $L$ whose closure $\bar C$ in $P_{[0,\infty]}$ is compact.
Since $\bar C\cap P_{[i,\infty]}$ is compact, so is its image $Q_i$ in $P_i$.
Let $K_i$ be the union of all closed cells of $L_i$ that intersect $Q_i$.
Since $P_i$ is locally compact, $K_i$ is a finite subcomplex of $L_i$.
Since each $Q_{i+1}$ maps into $Q_i$, each $K_{i+1}$ maps into $K_i$.
Since the cellulation of $L$ is based on the product cellulations $L_i\x[i-1,\,i]$,
we have $C\cap P_{(i-1,i)}\subset K_i\x(i-1,\,i)$.
Hence the mapping telescope $K_{[0,\infty)}$ contains $C$.
\end{proof}

\begin{proof}[(b)]
First we need to show that $\bar\nu'\subset\bar\nu$.
Every member $N_{[0,\infty)}$ of $\bar\nu'$ is a subcomplex of $L$ and hence a closed subset of $P_{[0,\infty)}$.
We need to show that it is also a closed subset of $P_{[0,\infty]}$.
Indeed, suppose that a sequence of points $x_i\in N_{[0,\infty)}$ converges to a point $x\in X$.
Let us recall that $P_{[0,\infty]}$ is the inverse limit of retractions $P_{[0,i+1]}\to P_{[0,i]}$ 
between the finite telescopes, which extend the bonding maps $p_i\:P_{i+1}\to P_i$ and send every
mapping cylinder $\cyl(p_i|_{\{x\}})\subset\cyl(p_i)=P_{[i,\,i+1]}$ onto $p_i(x)$.
Hence for each $n$ there is a retraction $q^\infty_n\:P_{[0,\infty]}\to P_{[0,n]}$, which extends 
the bonding maps $p^i_n\:P_i\to P_n$ and the limit map $p^\infty_n\:X\to P_n$
and sends every mapping cylinder $\cyl(p_i|_{\{x\}})\subset\cyl(p_i)=P_{[i,\,i+1]}$ onto $p^{i+1}_n(x)$.
Then the points $q^\infty_n(x_i)$ belong to $N_{[0,n]}$ and converge to $q^\infty_n(x)$, which
therefore also belongs to $N_{[0,n]}$.
Since it also belongs to $P_n$, it must be in $N_n$.
On the other hand, $x$ is nothing but the thread of the points $q^\infty_n(x)=p^\infty_n(x)$, $n=1,2,\dots$.
Hence $x\in\lim N_i$, which is a contradiction.

It remains to prove that each member of $\bar\nu$ lies in a member of $\bar\nu'$.
Let $C$ be a subcomplex of $L$ that is a closed subset of $P_{[0,\infty]}$.
Let $Q_i$ be the image of $C\cap P_{[i,\infty)}$ in $P_i$, and let $N_i$ be the union of all 
closed cells of $L_i$ that intersect $Q_i$.
Since each $Q_{i+1}$ maps into $Q_i$, each $N_{i+1}$ maps into $N_i$.
Since the cellulation of $P_{[0,\infty)}$ is based on the product cellulations $L_i\x[i-1,\,i]$,
we have $C\cap P_{(i-1,i)}\subset N_i\x(i-1,\,i)$.
Hence the mapping telescope $N_{[0,\infty)}$ contains $C$.

Suppose that $x=(x_0,x_1,\dots)\in\lim N_i$.
Then each $x_i$ is contained in the same cell of $L_i$ with a $y_i\in N_i$ such that there exist 
a $z_i\in N_j$, $j\ge i$, satisfying $p^j_i(z_i)=y_i$ and 
a $w_i\in\cyl(p_{j-1}|_{\{z_i\}})\subset\cyl(p_{j-1})=P_{[j-1,j]}$ such that $w_i\in C$.
Since $x_i\to x$ in $P_{[0,\infty]}$ and the cells of $L$ are scattered towards $X$, we also 
have $y_i\to x$ as $i\to\infty$.
Since $d(z_i,x)\le d(y_i,x)$, we also have $z_i\to x$.
But then also $w_i\to x$, so $C$ is not closed in $P_{[0,\infty]}$, which is a contradiction. 
\end{proof}

Proposition \ref{telescope filtration} and Corollary \ref{simplicial duality} yield
 
\begin{corollary} \label{telescope duality}
Let $S$ be a subcomplex of $L$.

(a) $S$ lies in a member of $\bar\nu'$ if and only if every member of $\kappa'$ meets $S$ in a finite subcomplex.

(b) $S$ lies in a member of $\kappa'$ if and only if every member of $\bar\nu'$ meets $S$ in a finite subcomplex.
\end{corollary}

The following consequence of Corollary \ref{telescope duality} admits a simple direct proof (bypassing 
\cite{M-I}*{Proposition \ref{fish:duality}}, Lemma \ref{simplicial cofinality} and 
Proposition \ref{telescope filtration}), 
which even works without our standing assumption that the $P_i$ are locally compact.

Let $\phi$ consist of the mapping telescopes $F_{[0,\infty)}$ of all restrictions $\dots\to F_1\to F_0$
of $\dots\to P_1\to P_0$ such that each $F_i$ is a finite subcomplex of $L_i$, and there exists an $n$ such that
$F_i=\emptyset$ for all $i\ge n$.

\begin{corollary} \label{telescope duality2}
Let $T_{[0,\infty)}$ be the mapping telescope of a restriction $\dots\to T_1\to T_0$ of $\dots\to P_1\to P_0$.

(a) $T_{[0,\infty)}\in\bar\nu'$ if and only if every member of $\kappa'$ meets $T_{[0,\infty)}$ 
in a member of $\phi$.

(b) $T_{[0,\infty)}\in\kappa'$ if and only if every member of $\bar\nu'$ meets $T_{[0,\infty)}$ 
in a member of $\phi$.
\end{corollary}

\begin{proof} If $K_{[0,\infty)}$ is a member of $\kappa'$ and $N_{[0,\infty)}$ is a member of $\bar\nu'$,
their intersection is the mapping telescope $F_{[0,\infty)}$ of a restriction $\dots\to F_1\to F_0$ of
$\dots\to P_1\to P_0$, where each $F_i=K_i\cap N_i$.
Then each $F_i$ is a finite complex and $\lim F_i=\emptyset$, which would not be the case if each 
$F_i\ne\emptyset$.
Thus $F_{[0,\infty)}$ is a member of $\phi$.

Next, if $T_{[0,\infty)}$ does not belong to $\bar\nu'$, then $\lim T_i$ contains a point $x=(x_0,x_1,\dots)$.
Let $K_i$ be the minimal closed cell of $L_i$ containing $x_i$.
Then each $K_{i+1}$ maps into $K_i$, and so $K_{[0,\infty)}$ is a member of $\kappa'\but\phi$
that lies in $T_{[0,\infty)}$.

Finally, if $T_{[0,\infty)}$ does not belong to $\kappa'$, then some $T_n$ is an infinite complex.
Then the inverse sequence $\dots\to\emptyset\to T_n\to\dots\to T_0$
is a member of $\bar\nu'\but\phi$ that lies in $T_{[0,\infty)}$.
\end{proof}

Let $A$ be a subcomplex of $L$ and $B$ a subcomplex of $A$.
Let $\nu'=\{\Cl(L\but F)\mid F\in\bar\nu'\}$ and
\[\resizebox{\linewidth}{!}{
\begin{minipage}{\linewidth}
\begin{alignat*}{4}
C_*^{\kappa'}(A,B)&\ =\ \colim_{K\in\kappa'}&C_*^\infty(A\cap K,\,B\cap K)&\ \simeq\ 
\colim_{K\in\kappa'}\ \lim_{N\in\nu'}\,C_*\big(A\cap K,\,(B\cap K)\cup (A\cap K\cap N)\big)\\
C^*_{\kappa'}(A,B)&\ =\ \lim_{K\in\kappa'}&C^*_c(A\cap K,\,B\cap K)&\ \simeq\ 
\lim_{K\in\kappa'}\ \colim_{N\in\nu'}\,C^*\big(A\cap K,\,(B\cap K)\cup (A\cap K\cap N)\big)\\
C_*^{\nu'}(A,B)&\ =\ \lim_{N\in\nu'}&C_*\big(A,\,B\cup (A\cap N)\big)&\ \simeq\ 
\lim_{N\in\nu'}\ \colim_{K\in\kappa'}\,C_*\big(A\cap K,\,(B\cap K)\cup (A\cap K\cap N)\big)\\
C^*_{\nu'}(A,B)&\ =\ \colim_{N\in\nu'}&\ C^*\big(A,\,B\cup (A\cap N)\big)&\ \simeq\ 
\colim_{N\in\nu'}\ \lim_{K\in\kappa'}\,C^*\big(A\cap K,\,(B\cap K)\cup (A\cap K\cap N)\big)
\end{alignat*}
\end{minipage}
}\]

Using the notation $C^*_X(A,B)$, $C_*^X(A,B)$ introduced in \S\ref{chain complexes},
Proposition \ref{telescope filtration} yields

\begin{corollary} \label{telescope complexes}
(a) $C_*^{\kappa'}(A,B)$ and $C_*^{\nu'}(A,B)$ are isomorphic to $C_*^X(A,B)$.

(b) $C^*_{\kappa'}(A,B)$ and $C^*_{\nu'}(A,B)$ are isomorphic to $C^*_X(A,B)$.
\end{corollary}

Let us recall that $H_i(X)\simeq H_{i+1}(P_{[0,\infty]},\,P_0)$ by direct arguments 
(see \S\ref{uniqueness-proof}) and 
$H_{i+1}(P_{[0,\infty]},\,P_0)\simeq H_{i+1}\big(C_*^X(P_{[0,\infty)},\,P_0)\big)$ by Theorem \ref{computation};
similarly for cohomology.
Hence we get

\begin{theorem} \label{telescope computation}
(a) $H_i(X)\simeq H_{i+1}\big(C_*^{\kappa'}(P_{[0,\infty)},\,P_0)\big)\simeq
H_{i+1}\big(C_*^{\nu'}(P_{[0,\infty)},\,P_0)\big)$.

(b) $H^i(X)\simeq H^{i+1}\big(C^*_{\kappa'}(P_{[0,\infty)},\,P_0)\big)\simeq 
H^{i+1}\big(C^*_{\nu'}(P_{[0,\infty)},\,P_0)\big)$.
\end{theorem}

\begin{remark} We can further modify the filtration $\kappa'$ so that the resulting filtration $\kappa''$ 
is indexed precisely by all compact $K\subset X$, and the resulting chain complex 
$C_*^{\kappa''}(P_{[0,\infty)},\,P_0)$ has the same homology as $C_*^{\kappa'}(P_{[0,\infty)},\,P_0)$ 
(and similarly for cochains).

Indeed, for each $i$ and each compact $K\subset X$, let $Q_{iK}$ be the union of all closed cells of $P_i$
that intersect the image of $K$ in $P_i$.
Since $P_i$ is locally compact, $Q_{iK}$ is a finite subcomplex of $P_i$.
By our choice of the triangulations of the $P_i$, each $Q_{i+1,K}$ maps into $Q_{iK}$.
The mapping telescope $Q_{[0,\infty),K}$ of the inverse sequence $\dots\to Q_{1K}\to Q_{0K}$ is a member of
$\kappa'$.
Let $\kappa''$ consist of the mapping telescopes $Q_{[0,\infty),K}$ corresponding to all compact $K\subset X$.

Given a member $Q_{[0,\infty)}$ of $\kappa'$, the inverse limit $K$ of $\dots\to Q_{1K}\to Q_{0K}$ is compact.
Since the image of $K$ in $P_i$ lies in $Q_i$, we have $Q_{iK}\subset Q_i$.
Moreover, it is easy to see that $Q_{[0,\infty)}$ deformation retracts onto $Q_{[0,\infty),K}$ by 
a vertical collapse.
\end{remark}

\begin{example} Suppose that $X$ is embedded in a cell complex $Q$ and we are given a sequence of iterated 
subdivisions $Q^{(i)}$ of $Q$ such that each cell of $Q^{(i)}$ has diameter $\le 2^{-i}$.
(For example, $Q=\R^m$ and $Q^{(i)}$ is the standard cubulation of $\R^m$ with vertex set $2^{-i}\Z^m$.)
Then we may take each $P_i$ to be the union of all closed cells of $Q^{(i)}$ that meet $X$, and the bonding maps
$P_{i+1}\to P_i$ to be the inclusions.
In this case we can further modify the filtration $\bar\nu'$ so that the resulting filtration $\bar\nu''$ is 
indexed by a cofinal family of neighborhoods of $X$ in $Q$, and the resulting chain complex 
$C_*^{\nu''}(P_{[0,\infty)},\,P_0)$ has the same homology as $C_*^{\nu'}(P_{[0,\infty)},\,P_0)$ 
(and similarly for cochains).

Indeed, given a neighborhood $U$ of $X$ that is a union of some of the closed cells of $Q,Q^{(1)},Q^{(2)},\dots$ 
that meet $X$, let $N_{iU}$ be the union of all closed cells of $Q^{(i)}$ that meet $X$ but do not lie in $U$.
Then each $N_{iU}\subset P_i$ and each $N_{i+1,U}\subset N_{iU}$.
Also we have $\bigcap_{i\in\N}N_{iU}=\emptyset$.
Indeed, $\bigcap_{i\in\N}N_{iU}\subset\bigcap_{i\in\N}P_i=X$, but for each $x\in X$ there exists an $n$
such that the closed $2^{-i}$-ball centered at $x$ lies in $U$, whence $x\notin N_{nU}$.
Thus the mapping telescope $N_{[0,\infty),U}$ of the inclusions $\dots\subset N_{1U}\subset N_{0U}$ is 
a member of $\bar\nu'$.
Let $\bar\nu''$ consist of the mapping telescopes $N_{[0,\infty),U}$ corresponding to all $U$ as above.

Given a member $N_{[0,\infty)}$ of $\bar\nu'$, let $U_i$ be the union of all closed cells of $Q^{(i)}$ 
that meet $X$ but do not lie in $N_i$, and let $U=\bigcup_{i\in\N} U_i$.
Then each $N_{iU}\subset N_i$, but if some cell $\sigma$ of $N_i$ contains no cells of $N_j$ for some $j>i$,
then $\sigma$ does not lie in $N_{iU}$. 
However, if $U'_i$ is the union of all closed cells of $Q^{(i)}$ that meet $X$ but do not lie in $N_{iU}$,
it is easy to see that the mapping telescope $U'_{[0,\infty)}=\Cl(P_{[0,\infty)}\but N_{[0,\infty),U})$ 
deformation retracts onto $U_{[0,\infty)}=\Cl(P_{[0,\infty)}\but N_{[0,\infty)})$ by a vertical collapse.
\end{example}

-----------------------------------

\begin{remark} We may use that $X$ is homotopy equivalent to $E_\Delta(X)$ by 
\cite{M-IV}*{???}. 
For the purposes of computation of homology the fibers of $E_\Delta(X)\to K_\Delta(X)$, which are compact subsets 
$K_\alpha$ of $X$, may be replaced by the appropriate compact pairs $(P_{\alpha,[0,\infty]},K_\alpha)$ in 
the extended telescope $(P_{[0,\infty]},X)$.
These can in turn be replaced by the one-point compactifications $(P_{\alpha,[0,\infty)}^+,\{\infty\})$
by the map excision axiom (note that we have a map with compact fibers $K_\alpha$).
It does not seem to be easy, however, to compute the homology of the resulting space (even though we know 
the answer).

If we want to deduce generalized additivity from something else, we can apply this construction to $(M,X)$.
Then the (co)homology of $(M,X)$ is isomorphic to the (co)homology of the ``telescope'' of clusters,
parametrized by $K_\Delta(X)$.
\end{remark}

Question: Do Milnor's additivity axioms for inseparable completely metrizable spaces imply the controlled additivity
for Polish spaces? Use the discretely indexed inseparable telescope, comupute its (co)homology by
inducting on skeleta. Maybe this can be used to deduce the continuity/compart support axiom from
Milnor's additivity axioms for inseparable completely metrizable spaces?

\subsection*{Acknowledgements}

I would like to thank P. Akhmetiev and T. Banakh for useful remarks.

\end{document}